\documentclass{article}
\usepackage{amssymb, amsthm, fullpage, amsmath, amscd,enumerate,wrapfig, mathrsfs, mathtools}
\usepackage{cite}

 \usepackage{graphicx, epsfig, caption, subcaption, tikz}
 \usepackage{epstopdf} 
\usepackage{color}

\usepackage{mathtools, hyperref}

\usepackage{accents}
\usepackage{url}
\usepackage[all]{xy}
\usepackage{psfrag}

\usepackage{tikz} 

\newtheorem{theorem}{Theorem}
\newtheorem{lem}[theorem]{Lemma}
\newtheorem{prop}[theorem]{Proposition}
\newtheorem{cor}[theorem]{Corollary}
\newtheorem{defn}[theorem]{Definition}

 \newcommand{\R}{\mathbb{R}}
  \newcommand{\Z}{\mathbb{Z}}
  \newcommand{\N}{\mathbb{N}}
   \newcommand{\C}{\mathbb{C}}
  \newcommand{\PR}{\mathbb{P}}

    \newcommand{\rk}{\textrm{rk}}
  
  \title{Limits Under Conjugacy of the Diagonal Subgroup in $SL_n(\R)$}
  \author{Arielle Leitner}
\begin{document}
  
  \maketitle 
  
  \begin{abstract} We give quadratic bounds on the dimension of the space of conjugacy classes of subgroups of $SL_n(\R)$
  that are limits under conjugacy of the diagonal subgroup. We give the first explicit examples of abelian $n-1$ dimensional subgroups of $SL_n(\R)$ which are \emph{not} such a  limit, and show all such abelian groups are limits of the diagonal group iff $n \leq 4$. 
  \end{abstract}

   %\begin{section}{Overview and Definitions}%%%%%%%%%%%%%%%%%%%%%%%%%%%%%%%%%%%%%%%%%%%%%%%%%%%%%%%
  
%  Let $G$ be a Lie group and $H$ a closed subgroup.  Recall a sequence of closed subgroups $H_n \leq G$ \emph{converges} to $H$ in the 
%  \emph{Chabauty topology} if the following two conditions are satisfied:\\
% (a) For every $h \in H$ there is a sequence $h_n\in H_n$ converging to $h$\\ 
% (b) For every sequence $h_n\in H_n$, if there is a subsequence which converges to $h$, then $h \in H$.   \\
% Recall a subgroup $L \leq G$ is called a  \emph{conjugacy limit} of a subgroup $H$, if there is a sequence of conjugating matrices, $P_n$, such that $P_n H P_n ^{-1}$ converges to $L$. 

Let $C \leq SL_n (\R)$ be the group of positive diagonal matrices, which is a Cartan subgroup.  The  limits of $C$ under conjugacy are classified for $n\leq 4$ in \cite{Haettel},  \cite{LeitnerSL3}, \cite{LeitnerGCP}.  It is an open problem to classify the conjugacy limits of $C$ when $n \geq 5$. 

  The set of all closed subgroups of a group is a Hausdorff topological space with the {\em Chabauty topology} on closed sets (see \cite{Chab}, \cite{Harpe}, \cite{Haettel}).  
  Following notation in \cite{IM},  the set of all closed abelian subgroups $\widehat{Ab}(n)= \{ G \leq SL_n(\R): G \cong (\R^{n-1}, +)\}$,  is a subspace, as is the set of conjugacy limit groups $\widehat{Red}(n) = \{ G \leq SL_n (\R): G \textrm { is a limit of } C \}$.  Taking the quotients by conjugacy, we have two topological spaces with the quotient topology:   $Ab(n) = \widehat{Ab}(n) /\textrm{conjugacy}$ and $Red(n) =\widehat{Red}(n)/ \textrm{conjugacy}$. In general these are not Hausdorff. For example, Theorem 16 in \cite{LeitnerSL3} shows $Red(2)= \{ C, P \}$, where $P$ is the parabolic group.  Since $C \to P$, every neighborhood of $P$ contains $C$.

 %Notice that $Red(n)$ is not Hausdorff, since the topology is given by convergence of the groups: $p$ is in a neighborhood of $q$ if there is a conjugacy limit $p \to q$.  Thus the Cartan subgroup is in a neighborhood of every other group. 
 
Every conjugacy limit of $C$ is isomorphic to $\R^{n-1}$, so $Red(n) \subset Ab(n)$ by \cite{IM}, Proposition 1.  From \cite{LeitnerSL3} and \cite{Haettel},
we know $Ab(3) = Red(3)$, which has 5 points corresponding to 5 conjugacy classes of groups,  and     $Ab(4) = Red(4) $, which has 15 points.  When $n\leq 6$, Suprenko and Tyshkevitch, \cite{SupTysh}, have classified maximal commutative nilpotent (i.e. $\textrm{ad}_x$ is nilpotent for all $x \in X$) subalgebras of $\mathfrak{sl}_n(\C)$.  Their results imply $Ab(5)$ has finitely many points, so $Red(5)$ has finitely many points.   Iliev and Manivel, \cite{IM}, ask if $Red(n)$ is finite when $n \geq 6$ (Question C).  The answer follows for $n\ge 7$ from the main result of this paper:
  
\begin{theorem}\label{maincov}  %If $n \geq 7$, then 
%$ \frac{n^2-6n}{8} 
%\frac{n^2 -8n +16}{8}
If $n \geq 7$, then
$ \frac{n^2 -8n +12}{8} \leq \dim Red(n) \leq n^2 -n$.
\end{theorem} 

%There are infinitely many conjugacy limits, and infinitely many abelian subgroups.
%but the set of limits of the Cartan subgroup is (in general) a proper subset of the set of abelian subgroups, see \cite{IM}.  
%There is a continuum of conjugacy classes of limit groups.    
The upper bound is given in \cite{IM}. This leaves the case $n=6$ open.
Haettel, and Iliev and Manivel show $\dim Red(n) < \dim Ab(n)$ for $n >6$.   We also give the first explicit examples of elements of $Ab(n) - Red(n)$ for $n=5,6,8$ by describing certain properties of limit groups, which answers Question A in \cite{IM}.  In particular, we show 

\begin{theorem}\label{mainex} If $n \leq 4$, then $Ab(n) = Red(n)$.  If $n \geq 5$, then $Red(n) \subsetneq Ab(n)$.
\end{theorem} 
  
 % \end{section}
  
  \begin{section}{A Family of Conjugacy Limit Groups}%%%%%%%%%%%%%%%%%%%%%%%%%%%%%%%%%%%%%%%%%%%%%%%
  
  In this section, we define a family of groups, $L_T$, and show each is a conjugacy limit of the Cartan subgroup.

%\begin{prop}\label{alpha_limit} For each $\alpha \in \R$, the group $L_{\alpha}$ is a conjugacy limit of the Cartan subgroup. 
%\end{prop}

%\begin{proof} 

%Let $\{P_n\}_{n=0}^\infty$ be the sequence of matrices 

%$$ P_n =  \left( \begin{array}{ccccccc}
%1 & 0 & 0 &0 &0  & n & 0 \\
%0 & 1 & 0 &0 &0 & n & n \\
%0&0&1&0&0& n & 2n \\
%0&0&0&1&0& n & \alpha n \\
%0&0&0&0&1 & n^2 & n^2\\
%0 &0&0&0&0&1&0\\
%0&0&0&0&0&0&1
% \end{array} \right). 
%$$

%Conjugating the diagonal Cartan group, $C = \textrm{Diag} \langle a,b,c,d,e,f,\frac{1}{abcdef} \rangle$, we have
%$$ P_n C P_n ^{-1}=
%\left( \begin{array}{ccccccc}
%1 & 0 & 0 &0 &0  & n(a-f) & 0 \\
%0 & 1 & 0 &0 &0 & n(b-f) & n(b-\frac{1}{abcdef}) \\
%0&0&1&0&0& n(c-f) & 2n (c-\frac{1}{abcdef}) \\
%0&0&0&1&0& n(d-f) & \alpha n (d-\frac{1}{abcdef}) \\
%0&0&0&0&1 & n^2(e-f) & n^2 (e-\frac{1}{abcdef})\\
%0 &0&0&0&0&1&0\\
%0&0&0&0&0&0&1
% \end{array} \right).
%$$

%We know $ P_n C P_n ^{-1}$ converges to a group with finite entries, so we start by giving names to some of the entries. Our goal is to find the relations between entries.   To see $ P_n C P_n ^{-1}$ converges to $L_{\alpha}$, define $s,t$ by $n^2(e-f) \to s$ and $n^2 (e-\frac{1}{abcdef})  \to t$.  Then $s-t = n^2 (f-\frac{1}{abcdef}) $  so $n (f-\frac{1}{abcdef})  \to 0$. In the second row, define $n(b-f) \to h$,  so that $n(b-\frac{1}{abcdef}) = n(b-f) - n(f-\frac{1}{abcdef}) \to h$.  The convergence in rows $3$ and $4$ is similar.  Thus the limit is $L _{\alpha}$.
%\end{proof} 

\begin{defn}\label{defLT}
Let $T$ be an $m$ by $n$ matrix, and $\rho_T : \R^{m+n} \to SL_{m+n+1} (\R)$ be the homomorphism given by 

$$\rho_T ( a_1, ...,  a_{m}, b_1, ..., b_n )=    
 {\left(\begin{array}{@{}c@{\quad}c}
 \underbrace{ {\begin{array}{ccc}
 \begin{matrix}
1 &  0 & \ldots &0\\
0 &  1& \ldots & 0\\
\vdots & \vdots & \ddots & \vdots\\
0 & 0& \ldots & 1
\end{matrix} 
    \end{array}}}_{(m+1) \times (m+1)}
 &
 {\begin{array}{ccc}
 \begin{matrix}
T_{11}a_1 &  T_{12}a_1  & \ldots & T_{1n}a_1\\
T_{21} a_2 &  T_{22} a_2& \ldots & T_{2n} a_2\\
\vdots & \vdots & \ddots & \vdots\\
T_{m1}a_{m}  &   T_{m2} a_{m}      &\ldots & T_{mn}a_{m}\\
b_1 & b_2 & \ldots & b_n
\end{matrix} 
    \end{array}}    \\ \\
 \underbrace{ {\begin{array}{ccc}
 \begin{matrix}
0 &  0 & \ldots &0\\
0 &  0& \ldots & 0\\
\vdots & \vdots & \ddots & \vdots\\
0 & 0& \ldots & 0
\end{matrix} 
    \end{array}}}_{n \times (m+1)}
     & 
    \underbrace{  {\begin{array}{ccc}
 \begin{matrix}
1 &  0 & \ldots &0\\
0 &  1& \ldots & 0\\
\vdots & \vdots & \ddots & \vdots\\
0 & 0& \ldots & 1
\end{matrix} 
    \end{array}}}_{n \times n}
  \end{array}\right)}. $$
  
The image of $\rho_T$ is a group, $ L_T \leq SL_{m+n+1} (\R)$. 

\end{defn} 

One may easily check that $\rho_T $ is a homomorphism and $L_T$ is a group, since matrix multiplication is given by 
$$\left( \begin{array}{c|c} 
I & P \\
\hline
0&I 
\end{array}\right) 
\left( \begin{array}{c|c} 
I & Q \\
\hline
0&I 
\end{array}\right)
=
\left( \begin{array}{c|c} 
I & P+Q \\
\hline
0&I 
\end{array}\right) . $$

\begin{lem}\label{islimit}  For any $m$ by $ n$ matrix $T$, with at least one nonzero entry in every row, the group $L_T$ is a conjugacy limit of the positive diagonal Cartan subgroup. 
\end{lem} 

\begin{proof}  Let $C = \textrm{diag}\langle x_1 , ... x_{m+n+1} \rangle \leq SL_{m+n+1}(\R)$, be the positive diagonal Cartan subgroup,  so $x_1 \cdot x_2 \cdot \cdot  \cdot x_{m+n+1}=1$.  Let $\{P_r\}_{r=0}^\infty$ be the sequence of matrices 

$$P_r= \left( 
 \begin{matrix}
1& 0 & \ldots &0&  0 & T_{11}r &  T_{12}r  & \ldots & T_{1n}r\\
0 & 1 & \ldots & 0& 0 & T_{21} r &  T_{22} r& \ldots & T_{2n} r\\
\vdots & \vdots & \ddots &\vdots &  \vdots & \vdots & \vdots & \ddots & \vdots\\
0&0& \ldots &1&  0 & T_{m1}r  &   T_{m2}r     &\ldots & T_{mn}r\\
0& 0& \ldots &0&  1 & r^2 & r^2 & \ldots & r^2\\
0& 0& \ldots  &0& 0& 1 & 0& \ldots &0 \\
0&0& \ldots & 0&0& 0&1 & \ldots &0\\
\vdots & \vdots & \ddots &\vdots &  \vdots & \vdots & \vdots & \ddots & \vdots\\
0&0& \ldots & 0&0& 0&0 & \ldots &1\\
\end{matrix} \right)
  $$

  %Conjugating $C$  by $P_r$, we have %(written in landscape for space/formatting reasons) 
  Conjugating, $P_r C P_r ^{-1} =$
  % \left(\begin{smallmatrix} A & B \\ 0 & D \end{smallmatrix} \right)$ where 
  
 % $$ A = \left ( \begin{matrix}
%  x_1 & 0 & \ldots & 0 & 0 \\
 % 0 & x_2 & \ldots 0 & 0 \\
 % \vdots & \vdots & \ddots & \vdots & \vdots \\
 % 0 & 0 & \ldots & x_m & 0 \\
 % 0 & 0 & \ldots & 0 & x_{m+1}  
%  \end{matrix} \right) ,
%  D = \left(  \begin{matrix} 
%  x_{m+2} & 0 & \ldots &0\\
%  0 & x_{m+3} & \ldots & 0 \\
%  \vdots & \vdots & \ddots & \vdots \\
%  0 & 0& \ldots & x_{m+n+1}
%  \end{matrix} \right) ,$$
  
%  $$B = \left( \begin{matrix} 
%  T_{11}r (x_{1}- x_{m+2})&  T_{12}r (x_{1}- x_{m+3}) & \ldots & T_{1n}r (x_1 - x_{m+n+1})\\
 % T_{21} r  (x_{2}- x_{m+2})&  T_{22} r (x_{2}- x_{m+3}) & \ldots & T_{2n} r  (x_2 - x_{m+n+1})\\
 %  \vdots & \vdots & \ddots & \vdots \\
  %T_{m1}r(x_{m}- x_{m+2})   &   T_{m2} r  (x_{m}- x_{m+3})    &\ldots & T_{mn} r  (x_{m} - x_{m+n+1})\\ 
 % r^2(x_{m+1}- x_{m+2}) & r^2 (x_{m+1}- x_{m+3}) & \ldots & r^2  (x_{m+1} - x_{m+n+1})
 % \end{matrix} \right ) .$$
  
 % The matrix $A$ is $m+1$ by $m+1$, and $D$ is $n$ by $n$.  The $0$ matrix is $m+1$ by $n$ and the matrix $B$ is $m+1$ by $n$. 
  
    %\begin{sideways}
  $$
 \bordermatrix{
 & & &  & & & \textrm{column } m+2 & \cr
&x_1& 0 & \ldots &0&  0 & T_{11}r (x_{1}- x_{m+2})&  T_{12}r (x_{1}- x_{m+3}) & \ldots & T_{1n}r (x_1 - x_{m+n+1}) \cr
&0 & x_2 & \ldots & 0& 0 &T_{21} r  (x_{2}- x_{m+2})&  T_{22} r (x_{2}- x_{m+3}) & \ldots & T_{2n} r  (x_2 - x_{m+n+1}) \cr
&\vdots & \vdots & \ddots &\vdots &  \vdots & \vdots & \vdots & \ddots & \vdots \cr
&0&0& \ldots &x_m &  0 & T_{m1}r(x_{m}- x_{m+2})   &   T_{m2} r  (x_{m}- x_{m+3})    &\ldots & T_{mn} r  (x_{m} - x_{m+n+1}) \cr
 \shortstack{row\\m+1}  &0& 0& \ldots &0&  x_{m+1} & r^2(x_{m+1}- x_{m+2}) & r^2 (x_{m+1}- x_{m+3}) & \ldots & r^2  (x_{m+1} - x_{m+n+1}) \cr
&0& 0& \ldots  &0& 0& x_{m+2} & 0& \ldots &0 \cr
&0&0& \ldots & 0&0& 0&x_{m+3} & \ldots &0 \cr
&\vdots & \vdots & \ddots &\vdots &  \vdots & \vdots & \vdots & \ddots & \vdots \cr
&0&0& \ldots & 0&0& 0&0 & \ldots &x_{m+n+1} } 
  $$
%\end{sideways}

Assume for simplicity that all entries in the first column of $T$ are non-zero. Given an element $l_T \in L_T$, we will find a sequence of elements in $ P_r C P_r ^{-1}$ which converges to $l_T$. % by proving relations between entries. 
%We will show that the same path of conjugating matrices $P_r$ may be used for convergence to all $l_T \in L_T$. 
Then the definition of convergence implies that %the entire group $ P_r C P_r ^{-1}$ converges to $L_T$.    
$L_T$ is a subgroup of the limit of $P_r C P_r ^{-1}$. 

%We will pick the sequence so that $x_i \to 1$ for all $i$.  
Given $x_{m+1}$ for $1\le i\le n$ define
$$x_{m+1+i} = -r^{-2} b_i + x_{m+1}.$$
This ensures row $m+1$ of $l_T$ and  of $P_r C P_r ^{-1}$ are equal since 
\begin{equation}\label{rowm1}r^2 (x_{m+1} - x_{m+1+i})= b_i. \end{equation} 
For $i\le m$ define $x_i$ in terms of $x_{m+1}$ by
$$x_i = r^{-1} a_i - r^{-2} b_1 + x_{m+1}.$$
It follows  that  column $m+2$ of $l_T$ and of $P_r C P_r ^{-1}$ are equal because 
\begin{equation}\label{col} x_{i}- x_{m+2}=(r^{-1} a_i - r^{-2} b_1 + x_{m+1})-(-r^{-2}b_1+x_{m+1})=r^{-1}a_i.\end{equation}
The determinant condition $x_1  \cdot \cdot \cdot x_{m+n+1}= 1$ determines $x_{m+1}$.
Observe that $x_i\to  x_{m+1}$ as $r\to\infty$, so the determinant is approximately $( x_{m+1})^{m+n+1}$. Thus
every $x_i\to 1$ as $r \to \infty$.

We have now determined $x_i$ for $1 \leq i \leq m+n+1$.  It remains to show convergence in the remainder of the entries.  Using equation \eqref{rowm1} since $r\to\infty$,     
$$r (x_{m+1} - x_{m+1+i})   \to 0.$$ 
By taking the difference of any two of these terms, $$r(x_{m+1+j}- x_{m+1+k}) \to 0, $$
and, in particular \begin{equation}\label{convgen}
 r(x_{m+2} - x_{m+1+k}) \to 0.
\end{equation}  
Consider the $( j, m+1+k)$ entry, for $1 \leq j,k \leq n$.  Using \eqref{col} and \eqref{convgen}, implies $$T_{jk}r (x_j - x_{m+k+1})= T_{jk}r (x_j - x_{m+2}) - T_{jk} r(x_{m+2} - x_{m+1+k}) \to T_{jk } a_j - T_{jk}0 = T_{jk}a_j .$$  

This completes the proof when the entries in the first column of $T$ are non-zero.  Suppose some entries in the first column of $T$ are zero.  By hypothesis, $T$ has a nonzero entry in every row, say $T_{jk}$. Pick $x_i$ for $1 \leq i \leq m$ so that $T_{jk}m (x_j - x_{m+1+k}) \to a_j T_{jk}$.  Since $T_{jk} \neq 0$, proceed as in the rest of the proof. Thus we have found a sequence  $\textrm{diag}\langle x_1 , ... x_{m+n+1} \rangle$ such that  $P_r C P_r ^{-1} \to l_T$. 
%This completes the proof when none of the entries in the first column of $T$ are zero. Take the closure of the set of conjugacy limits $\{L_T : T \textrm{ has no zero entries}\}$.  Since $Red(n)$ is closed, the closure of this set contains $\{ L_T : T \in \mathcal{T} \}$. This shows $L_T$ is a conjugacy limit group for any $T \in \mathcal{T}$. 

This shows $L_T$ is contained in the limit of $P_r C P_r^{-1}$. For dimension reasons (\cite{CDW} Proposition 3.1), and since $C$ and $L_T$ are connected and isomorphic to $\R ^{m+n}$ (see \cite{IM} Proposition 1), then $P_r C P_r ^{-1} \to L_T$.
%Given $l_T \in L_T$, we have found a sequence of $\textrm{diag}\langle x_1 , ... x_{m+n+1} \rangle$ which converges to $l_T$ under %conjugacy by $P_r$.  Notice that the same path of conjugating matrices, $P_r$, and the same method of finding a convergent sequence may %be used for any $l_T \in L_T$.  Therefore $P_r C P_r^{-1} \to L_T$. 
\end{proof} 

%\begin{rmk}\label{pbyq} We may exploit this idea to find unipotent conjugacy limit groups in $SL_{m+n+1}(\R)$ with a nonzero $m \times n$ coefficient matrix in the upper right corner, following the same pattern. 
%\end{rmk} 

\end{section}
  
  \begin{section}{A Continuum of  Conjugacy Classes of Limit Groups in $SL_7(\R)$}
  %%%%%%%%%%%%%%%%%%%%%%%%%%%%%%%%%%%%%%%%%%%%%%%%%%%%%

%Now that we have shown $L_T$ is a conjugacy limit group, we will determine the conjugacy classes of groups $L_T$, by considering the group action of $L_T$ on projective space.   

In this section we find some conjugacy invariants of the group $L_T$ and use them to produce a family of conjugacy classes
of dimension at least
$(n^2-8n+12)/8$ when $n \geq 7$. We first illustrate this when $n=7$. 

A subgroup $G \leq SL_n(\R)$ acts on $\R P^{n-1}$.  The orbit of a point, $x \in \R P^{n-1}$ is $G.x = \{ g.x :  g \in G  \}$.  Denote by $\overline{G.x}$ the \emph{orbit closure} of $x$.  %We say $\overline{G.x}$ is an orbit closure of $G$. 

\begin{lem}\label{conjorb} Suppose $G, H \leq SL_n(\R)$ %are conjugate subgroups by a similarity matrix $Q$, 
and $Q \in SL_n(\R)$ so that $G = Q H Q^{-1}$. 
Then $[Q]$ is a projective transformation taking the orbit closures of $G$ to the orbit closures of $H$. 
\end{lem}
\begin{proof}  %Suppose $G = Q H Q ^{-1}$, where $H \leq SL_n (\R)$.  View $Q$ as a projective transformation $Q : \R P^{n-1} \to \R P^{n-1}$.  
Since $Q$ conjugates $G$ to $H$, then $Q$ takes the orbits of $G$ to the orbits of $H$.  Hence $Q$ takes orbit closures of $G$ to orbit closures of $H$. 
%We claim $Q$ maps the orbit closures of $G$ to the orbit closures of $H$. Let $X$ be %a projective subspace that is 
%an orbit closure of $G$.  So $G.x \in X$ for all $x \in X$, and given $x,y \in X$, there exists $g \in G$ such that $g.x =y$.  Now, for any $g \in G$, we have $(Q g Q^{-1}).Q(x)   \in Q(X) $ for all $Q(x) \in Q(X)$, so $(Q G Q^{-1}).Q(x)   \in Q(X) $ for all $Q(x) \in Q(X)$.  Finally, for any $Q(x), Q(y) \in Q(X)$, there exists $g \in G $ such that $ QgQ^{-1} . Q(x) = Q(y)$. Thus $Q(X)$ is an orbit closure of $H$.  
%CONVERSE NOT TRUE!!!%Suppose $Q: \R P^{n-1} \to \R P^{n-1}$ takes the orbit closures of $G$ to the orbit closures of $H$.  Thus given any orbit closure, $X$,  of $G$, we know $Q(X)$ is an orbit closure of $H$.  For any $g \in G$,  and $x \in X$, we have $g.x \in X$.  This is also true for the orbit closures of $H$, so $Q g Q^{-1}.x \in Q.X $, or $Q g Q^{-1} \in H$ for all $g \in G$.  Thus $G, H$ are conjugate by $Q$. 
   \end{proof}

Given $G \leq SL_{n} (\R)$.  The \emph{orbit dimension function}, $\mathcal{R}_G: \R P^{n-1} \to \N$, is $\mathcal{R}_G(x)= \textrm{dim}(\overline{G.x})$.   As a corollary of Lemma \ref{conjorb},
$\mathcal{R}_G(Q(x)) = \mathcal{R}_{QGQ^{-1}} (x)$ for all $x \in \R P^{n-1}$.  

Next we define some conjugacy invariants of the action of a group on $\R P^{n-1}$.  To do this we need
 an invariant, the \emph{unordered generalized cross ratio},
  of a collection of points in general position in projective space, which generalizes the cross ratio of 4 points on a projective line. 
  This invariant is a finite subset of a product of projective spaces.
   Let $\mathcal{P}(S)$ denote the power set of $S$.

Let $\{e_1, ..., e_n \}$ be the \emph {standard basis} in $\R ^n$.  The \emph{standard projective basis} in $\R P^{n-1}$ is $\{ [e_1], ..., [e_{n}], [e_1 + \cdot \cdot \cdot +e_n]\}$, and
 an \emph{augmented basis} in $\R P^{n-1}$ is a set of $m \geq n+2$ points in general position, which means every subset of $(n+1)$ points is a projective basis.

\begin{defn} 
\begin{enumerate} 
\item The \emph{ordered generalized cross ratio} is the function, $\mathcal{C}: ( \R P^{n-1} ) ^{m} \to( \R P^{n-1})^{m-(n+1)}$ defined as follows.  Given any (ordered) augmented basis $(y_1, y_2, ..., y_{m})$ in $\R P^{n-1}$,  there is unique projective transformation, $Q$,  which maps $ (y_1, ..., y_{n+1}) \mapsto ([e_1], ...,[ e_{n}],[e_1 + \cdot \cdot \cdot +e_n]) $.  Define $\mathcal{C}( y_1, y_2, ..., y_{m}) := ( Q( y_{n+2}), Q(y_{n+3}), ..., Q(y_m) )$. 
\item Given an (unordered) augmented basis in $\R P^{n-1}$,  the \emph{unordered generalized cross ratio},  $\mathcal{UC} :(\R P^{n-1} )^m \to \mathcal{P}(( \R P^{n-1})^{m-(n+1)})  $ %$\mathcal{UC} : (\R P^{n-1})^m \to ( \R P^{n-1})^{(m-(n+1))m!}$, 
is the set of all generalized cross ratio tuples, $\mathcal{UC}(y_1, ..., y_{m}) :=\{\mathcal{C} ( y_{\sigma(1)}, ..., y_{\sigma(m)} )  : \sigma \in S_{m} \}$. 
\end{enumerate} 
\end{defn} 
For example, if $\mathcal{A} = \{ [ 1:0]: [1:1], [1:2], [1:\alpha] \}\subset \R P^1$, then 
$$\mathcal{UC}(\mathcal{A} ) = \big \{ \frac{2(\alpha-1)}{\alpha}, \frac{\alpha}{2(\alpha-1)}, \frac{\alpha}{2-\alpha}, \frac{2-\alpha}{\alpha}, \frac{2(\alpha -1)}{\alpha-2} , \frac{\alpha-2}{2(\alpha -1)}  \big \}\subset \R P^1.$$
Thus $\mathcal{UC}(\mathcal{A})$ is the set of all possible cross ratios of the points in $\mathcal{A}$.  %The cross ratio on $\R P^1$ is a special case of the ordered generalized cross ratio. 

\begin{prop}\label{CR} Let  $\{ y_1, ..., y_{m} \} $ and $ \{ x_1, ..., x_{m} \}$  be unordered augmented bases in $\R P^{n-1}$, so $m \geq n+2$.  Then  $\mathcal{UC}(y_1, ..., y_{m})  = \mathcal{UC}(x_1, ..., x_{m})$, if and only if there is a projective transformation, $Q: \R P^{n-1} \to  \R P^{n-1}$, such that  $Q( \{ y_1, ..., y_{m} \}) = \{ x_1, ..., x_{m} \}$. 
\end{prop} 
\begin{proof}  
First, suppose $\mathcal{UC}(y_1, ..., y_{m}) = \mathcal{UC}(x_1, ..., x_{m}) $.   For the generalized cross ratio tuple coming from the identity permutation, $\mathcal{C}(x_1, ...., x_{m}) \in \mathcal{UC} (x_1, ..., x_m)$, there is some reordering, $\sigma \in S_{m}$, such that  $\mathcal{C}(x_1, ...., x_{m}) = (z_1, ..., z_{m-n-1}) = \mathcal{C}( y_{\sigma(1)}, ..., y_{\sigma(m)}) \in \mathcal{UC} (y_1, ..., y_m)$.   That is, there exist projective transformations $Q_1 , Q_2 : \R P^{n-1} \to \R P^{n-1}$ such that $ Q_1 ( (x_1, ...., x_{n+1}))  = ([e_1], ...,[ e_{n}],[e_1 + \cdot \cdot \cdot +e_n])$ and $Q_2 (( y_{\sigma(1)}, ..., y_{\sigma(n+1)}) )= ([e_1], ...,[ e_{n}],[e_1 + \cdot \cdot \cdot +e_n])$, and also $Q_1 (x_{n+1+i}) =z_i = Q_2 (y_{\sigma(n+1+i)})$, for $1 \leq i \leq m-(n+1)$.  Set $Q : = Q_2 ^{-1} Q_1$, so $Q$ is a projective transformation such that $Q( (x_1, ...., x_{m}) )= 
( y_{\sigma(1)}, ..., y_{\sigma(m)})$.
% Given $z \in \mathcal{UC}\{x_1, ..., x_{n+2}\} $, there exists a projective transformation $Q: \R P^{n-1} \to \R P^{n-1}$, and $\sigma \in S_{n+2}$, such that $Q : (x_{\sigma(1)}, ..., x_{\sigma(n+1)}) \mapsto ([e_1], ...,[ e_{n}],[e_1 + \cdot \cdot \cdot +e_n]) $, and $Q(x_{\sigma(n+2)} )= z$.  Since $z \in \mathcal{UC}\{y_1, ..., y_{n+2}\} $, there is a projective transformation $P: \R P^{n-1} \to \R P^{n-1}$, and $\tau \in S_{n+2}$, such that $P : (x_{\tau(1)}, ..., x_{\tau(n+1)}) \mapsto ([e_1], ...,[ e_{n}],[e_1 + \cdot \cdot \cdot +e_n]) $, and $P(x_{\tau(n+2)} )= z$.  Let $M_\sigma$ denote a permutation matrix.  Then $Q M_\sigma M_\tau^{-1} P: \{ x_1, ..., x_{n+2} \} \mapsto \{ y_1, ..., y_{n+2} \}$. 

Conversely, suppose there exists a projective transformation $Q_0: \R P^{n-1} \to \R P^{n-1}$ such that \\
$Q_0( \{x_1, ...., x_{m}\}) =  \{ y_{1}, ..., y_{m}\}$.  Recall $\mathcal{UC}(x_1,...x_{m}) = \{ \mathcal{C} ( x_{\sigma(1)}, ..., x_{\sigma(m)} )  : \sigma \in S_{m} \}$.  Set $Q_\sigma : \R P^{n-1} \to \R P^{n-1}$ to be the unique projective transformation such that $Q_\sigma(  (x_{\sigma(1)}, ..., x_{\sigma(n+1)})) = ([e_1], ...,[ e_{n}],[e_1 + \cdot \cdot \cdot +e_n])$.  Then $\mathcal{UC} (x_1,...x_{m}) = \{ Q_\sigma (x_{\sigma(m)} ): \sigma \in S_{m}\}$.  Since $Q_\sigma Q_0^{-1} ( (y_{\sigma(1)}, ..., y_{\sigma(n+1)})) = ([e_1], ...,[ e_{n}],[e_1 + \cdot \cdot \cdot +e_n])$, and  such a projective transformation is unique, so $\mathcal{UC} (y_1,...y_{m}) = \{ Q_\sigma Q_0^{-1} (y_{\sigma(m)} ): \sigma \in S_{m}\} =  \mathcal{UC}(x_1, ..., x_{m}) $.
\end{proof} 

Proposition \ref{CR} shows that unordered cross ratio of an unordered augmented basis is a complete projective invariant. 
%, since the action of $S_4/K$ on the 4 points is a group of projective transformations. 
%We will compute a motivating example in $SL_7(\R)$ to 
As a warm-up, we show $Red(7)$ contains a subspace homeomorphic to an interval.   

 \begin{defn}  
Let $\alpha \in \R-\{0,1,2\}$ be fixed, and let  $\rho_{\alpha} : \R ^6 \to SL_7(\R)$ be the homomorphism defined by 
$$\rho_\alpha(  a,b,c,d,s,t) =
 \left( \begin{array}{ccccccc}
1 & 0 & 0 &0 &0  & a & 0 \\
0 & 1 & 0 &0 &0 & b & b \\
0&0&1&0&0& c & 2c \\
0&0&0&1&0& d & \alpha d \\
0&0&0&0&1 & e & f\\
0 &0&0&0&0&1&0\\
0&0&0&0&0&0&1
 \end{array} \right). 
$$
 The image of $\rho_{\alpha}$ is a group,  $L_{\alpha} \leq SL_7(\R)$. 

\end{defn}

An application of Lemma \ref{islimit} shows that $L_\alpha$ is a conjugacy limit group.  We use the unordered generalized cross ratio to distinguish conjugacy classes of limit groups. 

\begin{prop} Given $\alpha, \beta \in \R$, then $L_\alpha$ is conjugate to $L_\beta$ if and only if \\
 $\beta \in  \big \{ \frac{2(\alpha-1)}{\alpha}, \frac{\alpha}{2(\alpha-1)}, \frac{\alpha}{2-\alpha}, \frac{2-\alpha}{\alpha}, \frac{2(\alpha -1)}{\alpha-2} , \frac{\alpha-2}{2(\alpha -1)}  \big \}$.
\end{prop} 

\begin{proof} 
We showed in Lemma \ref{conjorb} that if two groups are conjugate, there is a projective transformation taking the orbit closures of the first group to the orbit closures of the second.  The group $L_{\alpha}$ partitions $\R P^6$ into orbit closures, and we will use the cross ratio to give an invariant of such a partition. 

 Let $\{e_1, ... e_7\}$ be the standard basis for $\R ^7 := V$. Let $ U= \langle e_1,..,e_5 \rangle$, and $W= \langle e_6, e_7\rangle $.  Then $V = U \oplus W$, and denote the quotient map $q : V \to V/U \cong W$.   Given $[t e_6 + e_7 ]\in \PR(W)$, define the 5-dimensional projective subspace $\mathcal{H}_t := \PR \langle e_1,... , e_5,  t e_6 + e_7 \rangle=\PR \langle q^{-1}(t e_6 +e_7) \rangle$. 
 We  show the orbit closure of a typical point $x \in \R P^6$ is some $\mathcal{H}_t$, but there are 4 exceptional $\mathcal{H}_t$, which are the pre-images of 4 points in $\PR(W)$.  The unordered  cross ratio gives an invariant of these  points in $\PR(W) \cong \R P^1$.
 
For convenience, denote the orbit dimension function for $L_\alpha$ by $\mathcal{R}_{\alpha} := \mathcal{R}_{L_{\alpha}}$.  Let $x = [x_1 : \cdot \cdot \cdot :x_7] \in \R P^6$.  The action of $L_\alpha$ is given by $L_\alpha . x=$
\begin{equation}\label{Lalph}
  [x_1 + ax_6: x_2+ b(x_6 + x_7): x_3 + c(x_6 + 2x_7): x_4+d(x_6 + \alpha x_7): x_5+ e x_6 + fx_7 :x_6:x_7].
\end{equation}

If $x \in \PR(U)$, then $\mathcal{R}_{\alpha}(x) = 0$, since $\PR (U) = \textrm{Fix} (L_{\alpha})$.  By \eqref{Lalph}, if $x \in \PR(V - U )$, then $ \mathcal{R}_{\alpha}(x) =5$, unless one or more of the coefficients on $a,b,c,d$ are zero, i.e., $x$ satisfies one of the equations 
\begin{equation}\label{sl7}
x_6 =0, \qquad x_6 + x_7 =0, \qquad x_6 + 2 x_7 =0, \qquad x_6 + \alpha x_7 =0.  
\end{equation}

Since $x \in V - U$, at least one of $x_6, x_7$ is not zero, and $x$ satisfies at most one equation in \eqref{sl7}.  Consequently, 
$$\mathcal{R}_\alpha(x) = 
  \begin{cases}
   0 & \text{if } x \in \PR(U)  \\
   4      & \text{if } x  \in \mathcal{H}_t  \textrm { and }  t \in \{ 0,1,2,\alpha\}\\
   5      & \text{if } x  \in \mathcal{H}_t  \textrm { and }  t \not \in \{ 0,1,2,\alpha\}.
  \end{cases}$$

%Thus there are 4 exceptional spaces $\{ \mathcal{H} _t : t = 0,1,2, \alpha\}$, which break into 4 dimensional orbit closures.    For all $t \not \in \{0,1,2, \alpha\}$, the subspace $\mathcal{H}_t$ is typical.  The exceptional subspaces are the pre-images, under $q$, of the points $\{ [te_6 + e_7] : t = 0,1,2, \alpha \}  \subset \PR(W)$. The cross ratio of $\{[0:1],[1:1],[2:1],[ \alpha:1] \}$ provides an invariant of the action of $L_{\alpha}$ on $\PR(W)$, up to the action of the group that permutes these points.  Since  $\mathcal{D}$ is the quotient of $\R P^1-\{0,1,\infty\} $ by $S_4 /K$,
% lemma \ref{conjorb} implies  $f$ is 1 to 1. 

Then $\mathcal{A} := \{ [1:t] \in \R P^1: t=0,1,2,\alpha \}$, is an augmented basis in $\R P^1$, and   %The unordered generalized cross ratio of  $\mathcal{A}$ is the set of cross ratios of $\mathcal{A}$, permuting the order of the points.  Thus 
$$\mathcal{UC}(\mathcal{A} ) = \big \{ \frac{2(\alpha-1)}{\alpha}, \frac{\alpha}{2(\alpha-1)}, \frac{\alpha}{2-\alpha}, \frac{2-\alpha}{\alpha}, \frac{2(\alpha -1)}{\alpha-2} , \frac{\alpha-2}{2(\alpha -1)}  \big \}\subset \R P^1.$$

 Therefore $L_\alpha$ is conjugate to $L_\beta$ if and only if $\beta \in \mathcal{UC}(\mathcal{A})$.
 \end{proof} 
    
We have shown the map $\R \to \widehat{Red}(7)$ given by $\alpha \to L_\alpha$ is at most 6 to 1.  Therefore $Red(7)$ contains a continuum of non-conjugate limits.  
 
  Recall the \emph{covering dimension} of a topological space, $X$, is smallest number, $n$, such that any open cover has a refinement in which no point is included in more than $n+1$ sets in the open cover.  (See \cite{Lurie}).  Denote the covering dimension of $X$ by $\dim X$. 
Covering dimension is a topological invariant.  We will show later that $\dim Red(7) \geq 1$. 
%Continuing with our example, let $K \leq S_4$ be the Klein 4 group, and $ D= S_4 /K \cong S_3$.  The action of $D$ on $\R - \{0,1, \infty\}$ is given by the action of $S_3$ on $\{0,1 \infty\}$, where the action on each component is determined by the action on its boundary.   Set $\mathcal{D} := \R - \{0,1, \infty\}/ D$, where we quotient by the action of $D$. Then $\mathcal{D}$ is homeomorphic to a half open interval.   We show in theorem \ref{homeo} that the map $\mathcal{D}  \to Red(7)$ is a homeomorphism to its image, so the covering dimension of $ \dim Red(7) \geq 1$, and $Red(7)$ contains a subspace homeomorphic to an interval.  

\end{section}

\begin{section}{The General Case: Bounds for $\dim Red(n)$}%%%%%%%%%%%%%%%%%%%%%%%%%%%%%%%%%%%%%%
%We define a new projective invariant of groups, which uses a generalization of the cross ratio. 

%the idea of taking cross ratios of dual points to orbit closures.  

In this section, we exploit the unordered generalized cross ratio to obtain bounds on $\dim Red(n)$ for $n \geq 7$.

 \begin{defn} Let $G \leq SL_n(\R)$ and $x \in \R P^{n-1}$.  Let $\mathcal{H} $ be a projective subspace of $\R P^{n-1}$. 
   \begin{enumerate} 
      \item  Set $M_G :=\textrm{max }\{ \mathcal{R}_G(x): x \in \R P^{n-1} \}.$
   \item  We call $x$ \emph{typical} if $\mathcal{R}_G(x) = M_G$. The subspace $\mathcal{H}$ is \emph{typical} if $\mathcal{H}$ is the orbit closure of a typical point. 
   \item   We say $x$ is \emph{exceptional} if $0< \mathcal{R}_G (x) < M_G$. The subspace $\mathcal{H}$ is \emph{exceptional} if $\mathcal{H}$ is the union of orbit closures of exceptional points, and $\textrm{dim } \mathcal{H} =M_G$. 
   \end{enumerate}
   \end{defn}
   
  Thus there are three types of points: fixed points with $\mathcal{R}_G (x) =0$, exceptional points when $0< \mathcal{R}_G (x) < M_G$, and typical points where $\mathcal{R}_G (x) = M_G$.  
  In our previous example, $M_{L_\alpha} = 5$ is the dimension of a typical subspace, and $\mathcal{H}_t$ is the orbit closure of a typical point.  There are 4 exceptional subspaces $\{ \mathcal{H}_t : t = 0,1,2, \alpha \}$ that break into orbit closures of smaller dimension.  Next we  generalize this example.

   \begin{defn}
An $m$ by $n$ matrix, $T$, is \emph{generic} if all collections of $n$ row vectors of $T$ are linearly independent. Set $\widehat{\mathcal{T} }: = \{ T \in \textrm{Mat}_{m \times n} : T \textrm{ is generic} \}$.   When $m \geq n+2$,  the rows of a generic matrix, $T$, determine an augmented basis,  $\tilde{T} \subset \R P^{n-1}$.  Define an equivalence relation on $\widehat{\mathcal{T}}$ by $T \sim S$ if $\mathcal{UC}(\tilde{T})=\mathcal{UC}(\tilde{S})$. Define $\mathcal{T} := \widehat{\mathcal{T} }/ \sim$, and denote by $[T]_{\mathcal{T}} \in \mathcal{T}$ the equivalence class of $T$. 
\end{defn} 

We give $\mathcal{T}$ a topology as follows. Take the subspace topology on $\widehat{\mathcal{T}}\subset \R ^{m \times n}$, then $\mathcal{T}$ has the quotient topology.  Since $\widehat{\mathcal{T}}$ is an open subset of $\R ^{n \times m}$, it follows $\dim \widehat{\mathcal{T}} = nm$. 

\begin{prop}\label{dimT} $\dim \mathcal{T} = nm - n^2 -m +1$. 
%nm - n^2 -m +1
%nm- n(n+1) -1.$
\end{prop} 
 \begin{proof}
 Consider the map $\Phi : \widehat {\mathcal{T}} \to (\R P^{n-1})^m$, where $\Phi(T) = \tilde T \in (\R P^{n-1})^m$, so $\Phi$ projectivizes the rows of $T$.   The unordered generalized cross ratio is the surjective map $\mathcal{UC}: (\R P^{n-1})^m \to (\R P^{n-1})^{m-(n+1)}$.  Given $T,S \in \widehat{\mathcal{T}}$, then $T \sim S$ if and only if $\mathcal{UC} (\Phi (T)) =\mathcal{UC} (\Phi (S)) $. 
The image of $\mathcal{UC}\circ \Phi$ is open since it consists of all sets of points in general position, and
$$\dim \mathcal{T}= \dim( \mathcal{UC} ( \Phi (\widehat{\mathcal{T}})))= (n-1) (m-n-1) = nm - n^2 -m +1.$$ 
%Where we subtract 1 for projectivizing. 
%$$\dim \mathcal{T}= \dim( \mathcal{UC} ( \Phi (\widehat{\mathcal{T}})))= (n-1) (m-n-1) = nm - n^2 -m+1.$$ 
 %
 % Denote the quotient map $q: \widehat{\mathcal{T}} \to \mathcal{T}$. Then $\dim \mathcal{T} = \dim \widehat{\mathcal{T}} - \dim \textrm{ker}(q)$.  Since $\sim$  allows us to map the projectivization of the first $n+1$ rows of $T \in \widehat{\mathcal{T}}$ to a projective basis, and to reorder rows, it follows $\dim \textrm{ker}(q) = n(n+1) +1$. 
 %
 %To find  $\dim \mathcal{T}$, take $\dim \widehat{\mathcal{T}}$, and subtract 1 for projectivizing. Subtract the dimensions corresponding to the quotient by the equivalence relation $\sim$, which allows us to map the projectivization of the first $n+1$ rows of $T \in \widehat{\mathcal{T}}$ to a projective basis, and reorder rows. 
\end{proof} 

Our earlier example, $ L _ \alpha \leq SL_7(\R)$, had a $2 \times 4$ matrix $T$, so $n=2$, and $m=4$.  We normalized by sending the first three rows to a projective basis of $\R P^1$, so $\dim Red(7) \geq 2\cdot 4 - 2^2 - 4+1 =1 $.
%2 \cdot 4 - 2 \cdot 3 -1 =1$. 

%\begin{defn}
%Suppose $G, G' \leq SL_n(\R)$.  Define an equivalence relation $G \sim G'$, if $G$ is conjugate to $G'$.   Let $\mathcal{C} (SL_n (\R))$ be the set of conjugacy classes of subgroups of $SL_n(\R)$, and denote the conjugacy class of $G$ by $[G] \in \mathcal{C} (SL_n (\R))$.
 %Write $SL_n (\R) / \sim$ for the special linear group quotiented by this relation.  Denote the conjugacy class of $L$ by $[L] \in SL_n (\R) / \sim$. 
%\end{defn}

%Write $[G ]$ for the conjugacy class of a subgroup $G \leq SL_n (\R)$. 

A set of hyperplanes is in \emph{general position} in $\R P ^n$, if the set of dual points in the dual projective space to these hyperplanes is in general position.  %Let $\delta: \R P^{n-1} \to \R P^{n-1}$ denote duality. 
Let $[L]$ denote the conjugacy class of a group $L$, and $T^t$ denote the transpose of $T$. 

%\begin{lem}\label{welldef} %For $m \geq n+2$, t
%The function $f : \mathcal{T} \to Red(m+n+1) $ given by $f: [T] \mapsto [L_T]$ is well-defined, where $[L]$ denotes the conjugacy class of $L$. \end{lem}

\begin{prop}\label{infinitelymany} Suppose  $m \geq n+2$, and $n \geq 2$. The function $f : \mathcal{T} \to Red(m+n+1) $ given by $f([T]_{\mathcal{T}})=[L_T]$ is  well defined and injective.  \end{prop} 

\begin{proof}  
First we show $f$ is well defined.  Suppose $[S]_{\mathcal{T}} = [T]_{\mathcal{T}}$.  Then there is a linear map $Q : \R ^n \to \R^n$ such that $Q$ maps the rows of $T$ to the rows of $S$.  That is, $Q( T^t )= S^t$, and taking the transpose of both sides, $T Q^t = S$.  Set $Q^t =P$.  Then $L_T$ is conjugate to $L_S$ by $ I_{m+1} \oplus P^{-1}$, because: 

$$\left( \begin{array}{c|c}
I & 0 \\
\hline
0 & P^{-1}
\end{array} \right) 
\left( \begin{array}{c|c}
I & T \\
\hline
0 & I
\end{array} \right) 
\left( \begin{array}{c|c}
I & 0 \\
\hline
0 & P^{-1} 
\end{array} \right) ^{-1} 
= 
\left( \begin{array}{c|c}
I & TP \\
\hline
0 & I
\end{array} \right) = 
\left( \begin{array}{c|c}
I & S \\
\hline
0 & I
\end{array} \right)
$$
 So if $[T]_{\mathcal{T}} = [S]_{\mathcal{T}}$ then $[L_T] = [L_S]$.  This shows $f$ is well-defined.

To prove $f$ is injective, we show if $[T]_\mathcal{T} \neq [S]_\mathcal{T}$, then the actions of $[L_T]$ and $[L_S]$ partition $\R P^{m+n}$ into orbit closures which are not projectively equivalent.  %We give an invariant of such a partition, which shows if $[T]_{\mathcal{T}} \neq [S]_{\mathcal{T}}$ then $[ L_T] \neq [L_S]$. 

Let $\{ e_1 , ... e_{m+n+1}\}$ be the standard basis for $V= \R ^{m+n+1}$.  Define %$V = \langle e_1, ..., e_{m+n+1} \rangle$, 
$U = \langle e_1 , ... , e_{m+1} \rangle$, and $W = \langle e_{m+2} ,... , e_{m+n+1} \rangle$, then $V = U \oplus W$. Let $q: V \to V/ U \cong W$ be the quotient map. Given $[v] \in \PR(W)$, let $\mathcal{H}_v$ be the $(m+1)$-dimensional projective subspace $\mathcal{H}_v = \PR \langle e_1, ... , e_{m+1}, v  \rangle =  \PR \langle q^{-1} (v) \rangle $.  We show the orbit closure of a typical point $x \in \R P^{m+n} $ is $\mathcal{H}_v$, and the exceptional subspaces are the pre-image of $m$ hyperplanes in $\PR(W)$, which determine an invariant of $L_T$.

The orbit dimension function for $L_T$ by $\mathcal{R}_{T} := \mathcal{R}_{L_T}$, has maximum $M_T:= M_{L_T}$.  The action of $L_T$ on $\R P^{m+n}$ is given by 
\begin{equation}\label{LTact}
\begin{array}{l}
L_T . [x_1 : \cdot \cdot \cdot : x_{m+n+1}]= \\
 \displaystyle 
 [x_1 + a_1 (\sum _{i=1}^n T_{1i} x _ {m+1+i}) :x_2 + a_2 (\sum _{i=1}^n T_{2i} x _ {m+1+i})  :\cdot \cdot \cdot : \\
  \displaystyle  x_{m} + a_m( \sum _{i=1}^n T_{mi} x _ {m+1+i}) :x_{m+1} +\sum_{i=1}^n  x_{m+1+i}b_i: x_{m+2}: \cdot \cdot \cdot :x_{m+n+1} ].
  \end{array}
 \end{equation}
 
 Set 
\begin{equation}\label{phi}   \phi_j (x_{m+2}, ..., x_{m+n+1}) =  \sum _{i=1}^n T_{ji} x _ {m+1+i}, \hspace{.5in} 1 \leq j \leq m,\end{equation}
a collection of linear functionals $\phi_j : \R^{n} \to \R$. Then we may rewrite 
\begin{equation}\label{LT}
\begin{array}{l}
L_T . [x_1 : \cdot \cdot \cdot : x_{m+n+1}]= \\
 \displaystyle 
 [x_1 + a_1  \phi_1 (x_{m+2}, ..., x_{m+n+1}) :x_2 + a_2  \phi_2 (x_{m+2}, ..., x_{m+n+1})  :\cdot \cdot \cdot : \\
  \displaystyle  x_{m} + a_m  \phi_n (x_{m+2}, ..., x_{m+n+1}) :x_{m+1} +\sum_{i=1}^n  x_{m+1+i}b_i: x_{m+2}: \cdot \cdot \cdot :x_{m+n+1} ].
 \end{array}
 \end{equation}

Since $T \in \mathcal{T}$ is generic, any $n$ rows of $T$ are linear independent, so by \eqref{LT}, $M_T = m+1$.  %It is easy to see from definition \ref{defLT} that 
If $x \in \PR (U)$, then the group $L_T$ fixes $x$, so $\mathcal{R}_T (x) =0$.  %since $L_T$ acts as the identity on $\PR (U) = \textrm{Fix} (L_T)$.  
We want to find the exceptional points.  % By \eqref{LT}, $\mathcal{R}_T(x) < m+1$ if and only if  the coefficient on some $a_i$ is zero.  Since the $\phi_i$ are linear independent, and no $\phi_i$ is zero, the coefficient on $a_i$ is zero if and only if $(x_{m+2}, ..., x_{m+n+1})$  is in the kernel of some of the linear functionals, $\phi_i $.  
From \eqref{LT} the coefficient on $a_i$ is $\phi_i$.  Thus $\mathcal{R}_T (x) < m+1$ if and only if $\phi _i $ is zero, i.e., $(x_{m+2}, ..., x_{m+n+1}) \in \textrm{ker} (\phi_i)$.

The set $W_j := \textrm{ker}(\phi_j)  \subset W$ is a hyperplane. Then $\mathcal{R}_T (x) < m+1 $  if and only if $x \in q^{-1} (W_j ) = U \oplus W_j$, for some $1 \leq j \leq m$. Thus, the set of exceptional points is the pre-image of the $m$ hyperplanes, $\PR(W_j) \subset \PR(W) \cong \R P^{n-1}$. Let $w_j\in\PR(W^*)$ denote the point in the dual projective space determined by the hyperplane $W_j\subset W$.

By hypothesis, $T \in \mathcal{T}$ is generic, so these hyperplanes are in general position.  The  points $\{w_j\}_{j=1}^m$ are in general position, and form an augmented basis, 
\begin{equation}\label{delT} \delta(T)\equiv \{w_1,\cdots,w_m\} \subset \PR(W^*)\cong\R P^{n-1} \end{equation} 

We are now able finish the proof that $f$ is injective.  Suppose $[T]_\mathcal{T},[S]_\mathcal{T} \in \mathcal{T}$ with $f([S]_\mathcal{T}) =f([T]_\mathcal{T})$.   That is, $L_S$ is conjugate to $L_T$, so Lemma \ref{conjorb} implies this conjugacy takes the exceptional hyperplanes in the orbit closures of $L_T$, to the exceptional hyperplanes in the orbit closures of $L_S$.  The dual conjugacy takes the dual augmented basis, $\delta(T)$, to the dual augmented basis, $\delta(S)$.   By Proposition \ref{CR}, $\mathcal{UC}(\delta(T)) = \mathcal{UC} (\delta(S))$, so there is a projective transformation taking $\delta(T)$ to $\delta (S)$.   A  row of $T$  determines a dual vector, $\phi_i$, with $\ker\phi_i=W_i$, dual to $w_i=[\phi_i]\in  \delta (T)$. So the dual transformation takes the (projectivized) rows of $T$ to the (projectivized) rows of $S$.  Thus $[T]_{\mathcal{T}} = [S]_{\mathcal{T}}$, and $f$ is injective. 
%The unordered generalized cross ratio, $\mathcal{UC}(\tilde{T})$, is an invariant which distinguishes the equivalence class of $[T]$: by proposition \ref{CR}, $\mathcal{UC}(\tilde{T}) = \mathcal{UC}(\tilde{S})$ if and only if $[T] = [S]$. Thus there is no projective transformation in $\R P^{n-1}= \PR ( W )$ taking  $\tilde{T}$  to $\tilde{S}$ for $[T] \neq [S]$, and hence no projective transformation of $\R P^{2n+2} = \PR(V ) $ taking the orbit closures of $[L_T]$ to the orbit closures of $[L_S]$. Therefore lemma \ref{conjorb} implies $f$ is 1 to 1. 
\end{proof} 

Proposition \ref{infinitelymany} shows there are infinitely many non-conjugate limits of the positive diagonal Cartan subgroup in $SL_k(\R)$ when $k \geq 7$.  We want to give bounds for $\dim Red(k)$.  %Denote by $\mathcal{P}(S)$ the power set of $S$. 
In the remainder of the section, set $k=m+n+1$. 

\begin{theorem}\label{homeo} Let $m -2 \geq n \geq 2$. %The function $f : \mathcal{T} \to Red(m+n+1) $ given by $f([T]_{\mathcal{T}} )= [L_T]$ restricted to
%any compact set $K$ is a homeomorphism onto its image.   
The function $\hat f: \widehat{\mathcal{T}} \to \widehat {Red} (k)$ defined by $\hat{f} (T) = L_T$, is  continuous and one to one on an open subset, $X \subset \widehat{\mathcal{T}}$, and $\dim X = \dim \mathcal{T}$. 
\end{theorem}

\begin{proof} 
%View the image $\hat f (\widehat{\mathcal{T}})\subset \widehat{Red}(k)$ as endowed with the topology of the Grassmannians.  Then $\hat{f} (T) - I_k = L_T - I_k$ is a linear map, so $\hat f$ is an affine map, and therefore $\hat f$ is continuous. 

Recall from Definition \ref{defLT} the linear map $\rho_T : \R ^{m+n} \to SL_{m+n+1} (\R) \subset \textrm{End}(\R ^{m+n+1})$.  Thus $\rho_T \in \textrm{Hom}(\R ^{m+n}, \textrm{Mat}_{m+n+1})$.  Since $T \in \R ^{m+n} = \textrm{Hom}(\R ^n, \R ^m)$, the map $ T \mapsto \rho_T$ is a continuous linear map, as it maps one matrix to a larger one.   Remember $L_T$ is defined as the image of $\rho_T$, so we view $L_T \subset \textrm{End}(\R ^{k})$.    Since $\hat{f} (T) = L_T$, the image $\hat{f} (\widehat{\mathcal{T}} ) \subset  \textrm{End}(\R ^{k})$, and $\hat{f}$ is continuous. 

Define $X \subset \widehat {\mathcal{T}}$ as follows.  The symmetric group $S_{m-n-1}$ acts on $\textrm{Mat}_{(m-n-1) \times n}$ by permuting rows.   Pick a matrix $A_0 \in \textrm{Mat}_{(m-n-1) \times n}$ so that $S_{m-n-1}$ acts freely on the orbit of $A_0$.   Let $\mathcal{N}$ be a tiny neighborhood around $A_0$.  Denote $v = (1, ..,1)$ a row vector. Set 
$$X := \{  T= \left (\begin{array}{c}I_n \\ v \\ A \end{array} \right) : A \in \mathcal{N} \} \subset \widehat{\mathcal{T}}.$$ 
Then $X$ is the collection of $m \times n$ matrices where the first $n+1$ rows have a projectivization that is the standard basis for $\R P^{n-1}$, and the projectivization of the last $m-n-1$ rows is in general position in $\R P^{n-1}$ (because $T$ is generic).   Since $\mathcal{N}$ is a tiny neighborhood around $A$, no two points in $X$ have the same unordered generalized cross ratio. Thus $X$ contains one representative of each equivalence class in $\widehat{\mathcal{T}}$, and $f: X \to \widehat{Red}(k)$ is one to one.   
By the same argument as in the proof of Proposition \ref{dimT}, 
$$\dim X= \dim( \mathcal{UC} ( \Phi (X)))= (n-1) (m-n-1) = nm - n^2 -m +1 = \dim \mathcal{T}.$$  
\end{proof}

\begin{cor}\label{lb} If $k \geq 7$, then $\dim Red(k)  \geq \frac{k^2 -8k +12}{8}$.
% \frac{k^2 -8k +16}{8}
%\frac{k^2 -6k}{8}$. 
%$Red(m+n+1)$ is bounded below by $nm- n(n+1) -1$
\end{cor} 
\begin{proof} %Set $k = m+n +1$.  
%Covering dimension is a topological invariant, preserved by homeomorphism. 
 Proposition \ref{dimT} says  $\dim X = nm- n^2 -m +1$, and Theorem \ref{homeo} %shows $\mathcal{T}$ is homeomorphic to a subspace of $Red(m+n+1)$ if $m-2 \geq n \geq 2$.   
implies $\dim \widehat{Red}(k) \geq \dim X$. 
Since $k \geq 7$, we may choose $m-2 \geq n \geq 2$.   %If $k \leq 6$, the result is vacuous.

 We may change the size of the $m$ by $n$ matrix (as long as $m-2 \geq n \geq 2$), so $\dim Red(k)$ is bounded below by the maximum of %$nm - n(n+1) -1$.  
 $mn - n^2 -m +1$. 
 Since $m+n+1=k$, and $k$ is fixed, we want to maximize % $g(n) = n (k-n-1) - n(n+1) -1$. 
 $$g(n) = n(k-n-1) - n^2 - (k-n-1)+1  = kn - 2n^2 -k +2.$$ 
  The maximum occurs at %$n = \frac{k-2}{4}$, so $m = \frac{3k-6}{4}$, 
 $n= \frac{k}{4}$ and $m= \frac{3k-4}{4}$. 
 %and the maximum of $mn - n(n+1) -1$ is $\frac{k^2 -6k}{8}$. 
% Therefore the maximum of  $mn - n^2 -m +1$ is $\frac{k^2-8k +8}{8}$. 
However, we need $m,n \in \Z$.  An easy computation with $k \equiv 0,1,2,3 \pmod 4$ shows 
%\begin{description} 
%\item[$k \equiv 0 \pmod 4$]  Suppose $k = 4l$.  Then the maximum occurs at $n=l$ and $m=3l -1$, so %the maximum of $g$ is $2l^2 -4l = \frac{k^2 -8k +16}{8} \in \Z$. 
%\item[$k \equiv 1 \pmod 4$] Suppose $k = 4l+1$.  Then the maximum occurs at $n=l$ and $m=3l $, so %the maximum of $g$ is $2l^2 -3l +1=\frac{k^2 -8k +15}{8} \in \Z$. 
%\item[$k \equiv 2 \pmod 4$] Suppose $k = 4l+2$.  Then the maximum occurs at $n=l$ and $m=3l +1$, %so the maximum of $g$ is $2l^2 - 2l =\frac{k^2 -8k +12}{8} \in \Z $. 
%\item[$k \equiv 3 \pmod 4$] Suppose $k = 4l+3$.  Then the maximum occurs at $n=l+1$ and $m=3l %+1$, so the maximum of $g$ is $2l ^2 -l =\frac{k^2 -8k +15}{8} \in \Z$. 
%\end{description} 
%Thus for any $k \geq 7$, 
the maximum of $nm - n^2 +m +1$ subject to $k =m+n+1$ for $m,n \in \Z$ is at least $\frac{k^2 -8k +12}{8}$. 
\end{proof}

In particular, $\frac{k^2 -8k +12}{8} > 0$ for $k \geq 7$. Below is the proof of an upper bound of $\dim Red(k)$, given in \cite{IM} for (Krull) dimension of $Red(k)$. 

\begin{theorem}\label{ub} $\dim Red(k) \leq k^2-k$. 
\end{theorem} 
\begin{proof}  Let $C$ denote the positive diagonal Cartan subgroup, and let $P \in GL_k (\R)$.  By \cite{OV} Theorem 1, or \cite{Vard} Theorem 2.9.7,  the dimension of the set of all conjugates of $C$ is $k^2 -k$, since $P C P^{-1} = C$ if and only if $P$ is a diagonal matrix or a permutation matrix.     
Since $C$ is a semi-algebraic set (\cite{BCR} Proposition 2.1.8), the set of conjugates of $C$ is a semi-algebraic set (\cite{BCR} Proposition 2.2.7).  Thus the set of conjugacy limits of $C$ is the boundary of the Zariski closure of the set of conjugates.  Applying  Propositions 2.8.2 and 2.8.13 from \cite{BCR}, gives $\textrm{dim} (Red(k)) \leq k^2 -k$. 
\end{proof} 

Corollary \ref{lb} and Theorem \ref{ub} imply Theorem \ref{maincov}.

%\begin{proof}  Set $n= 2m$ in theorem \ref{lb}, since $m -2 \leq 2m$ as soon as $m\geq 2$. Thus for $k \geq 7$, and $k \equiv 1 \pmod 3$ we have $m + 2m +1 =k$.  If $k \equiv 0,2 \pmod 3$, set $n = 2m \pm 1$. Thus corollary \ref{lb} gives a quadratic lower bound, and theorem \ref{ub} gives a quadratic upper bound. 
%\end{proof} 

\end{section}

\begin{section}{Abelian Groups which are Not Conjugacy Limit Groups}%%%%%%%%%%%%%%%%%%%%%%%%%%%%%%%%%%%%%%%%%%%%%%%%%%%%

In this section, we give examples of elements of $Ab(n ) - Red(n)$.  There are two properties of conjugacy limit groups of $C$ 
which are not universal amongst abelian groups. The first property is a conjugacy limit group is {\em flat}, and
the
second is that it contains a one parameter subgroup with a particular Jordan block structure.

%We will show conjugacy limit groups of the Cartan subgroup are isomorphic to flat planes away from the origin, and contain a rank 1 subgroup. These examples fail to satisfy those properties of conjugacy limit groups.  

  Suppose $L$ is a conjugacy limit of $C$ in $SL_n( \R)$.   Then we claim $L$ is the intersection of a vector space with $SL_n(\R) \subset \textrm{End}(\R ^n),$ which is a vector space.  Such a group is a \emph{flat group}. 
 The positive diagonal Cartan subgroup is flat, and conjugacy is a linear map, so it preserves this property. 
 Thus conjugacy limits of $C$ are flat groups. 

\begin{defn}  Let $\mu_k : \R ^{k-1} \to SL_k (\R)$ be the representations below for $k = 5,6$. 
$$\mu_5 ( a, b,c,d) =
\left ( \begin{array}{ccccc}
1 & a & 0 & \frac{a^2}{2} & b\\
0 & 1 & 0 & a&0\\
0 & 0 & 1 & c & d\\
0&0&0& 1 & 0 \\
0&0&0&0&1
\end{array} \right) ,
\mu_6 ( a,b,c,d,e) =
\left( \begin{array}{cccccc}
1 & a & \frac{ a^2}{2} & 0 & b & c\\
0 & 1 & a & 0 & 0 & 0 \\
0 & 0& 1 & 0 & 0 & 0\\
0 & 0& 0& 1 & d & e\\
0 & 0& 0 & 0 & 1 & 0\\
0 & 0& 0& 0& 0& 1
\end{array} \right) 
$$
 
 Set $M_k \leq SL_k( \R)$ to be the respective images of $\mu_k$. 
\end{defn} 

It is easy to check that $M_k$ is an abelian group of dimension $k -1$.   Moreover, neither is a limit of  $C$, since these are not flat groups. 

%\begin{cor} The groups $M_k$ are not conjugacy limits of the Cartan subgroup. 
%\end{cor} 
%\begin{proof} By proposition \ref{linsubsp} any conjugacy limit of the Cartan subgroup is a linear %subspace of $ \R ^{n^2-1} - \{0\}$.  Neither $M_5 $ or $M_6$ has this property. 
%\end{proof} 

Thus we have given examples of elements in $Ab(n) - Red (n) $ for $n = 5,6$.  This shows $Ab(n ) \neq Red(n)$ when $n=5,6$, which answers Question A in \cite{IM}. %Haettel, \cite{Haettel}, proves in lemma 3.4 that for $n \geq 7$, there is an abelian subalgebra of dimension $n-1$ which is not a conjugacy limit of diagonal Cartan subalgebras in $SL_n(\R)$, following the argument in \cite{IM} for the complex case. 
By Proposition 2 in \cite{IM} or Lemma 3.4 in \cite{Haettel}, there is an abelian subalgebra of dimension $n-1$ which is not the conjugacy limit of a Cartan subalgebra. 
Combining these results implies $Ab(n) = Red(n)$ if and only if $n \leq 4$.  For $n = 5,6$, we have shown $Red(n)\subsetneq  Ab(n)$. This completes the proof of Theorem \ref{mainex}. 

We give another property satisfied by conjugacy limit groups of $C$, and an example of an element of $Ab(8) - Red(8) $, which is a flat group, but does not satisfy this additional property.    Thus to determine if a group is a conjugacy limit of $C$, it is necessary but not sufficient for the group to be a flat group.

 Suppose $\R^{n-1} \cong G \leq SL_n(\R)$, and let $\rk (A)$ denote the rank of a matrix $A$.   Define the \emph{tier} of $G$ to be $\displaystyle \textrm{tier}(G) = \tau (G) := \max _{g \in G} \rk(g - I_n)$. If $\rk (g - I_n) = \tau (G)$, then $g$ is \emph{generic}. 
In the special case when $ G$ is a unipotent group, one may compute the tier from the Jordan Normal Form (JNF) of a generic group element, by counting the number of off-diagonal entries.

  \begin{prop}\label{rank} Suppose $G \leq SL_n(\R)$ is a unipotent group, 
%$G \in \hat{Red}(n)$ is a limit group, 
and $L$ is a conjugacy limit of $G$.  Then $\tau (L) \leq \tau (G)$. 
  \end{prop}

  \begin{proof} 
  Given $l \in L$, let $(g_m)$ be a sequence of conjugates of $g \in G$ that converges to $l$.  Since $G$ and $L$ are unipotent, $g - I_n$ and $l -I_n$ are are conjugate to strictly upper triangular matrices.  Passing to a convergent subsequence of $(g_m)$, we may assume the rank of $g_m$ is constant.  Since rank is lower semi-continous, 
  $$ \liminf_{m \to \infty} \rk (g_m -I_n)  \geq \rk (l -I_n)$$ 
  and the result follows. 
  \end{proof}
  
%Let $G \leq SL_n(\R)$ be the image of a representation $\rho : \R ^{n-1} \to SL_n(\R)$.  Then $G$ contains a \emph{flag of subgroups}, $H_i \leq G$, if the following conditions are satisfied: $ H_{i-1} \leq H_i$, and each $H_i$ is the image of $\R ^i$ under $\rho$. 

Suppose $G \leq SL_n(\R)$ is isomorphic to $(\R ^{n-1},+)$.  A \emph{flag of subgroups} in $G$ is a collection of subgroups $H_i \cong ( \R ^i , +)$ with $1 \leq i \leq n-1$, and $H_{i-1} \leq H_i$. 
  
  \begin{cor}\label{ranklim} If $L$ is a conjugacy limit of $C$, then $L$ contains a flag of subgroups, $H_i$, with %rank less than or equal to $1, ..., \rk(L)$. 
 $ \tau ( H_i) \leq i$ for all $i$.  In particular, %every conjugacy limit of $C$ contains a 1 parameter subgroup with rank 1. 
 $L$ contains a 1 parameter subgroup $H_1$ with $ \tau (H_1) =1$. 
  \end{cor} 
  \begin{proof} Suppose $P C P^{-1} \to L$ by some $P \in SL_n(\R)$.  Set $C_1 =  \textrm{diag} \langle a, 1,1,...,1 \rangle ,$ and let $L_1$ be the conjugacy limit of $C_1$ by $P$.  By Proposition \ref{rank}, $\tau ( L_1) \leq \tau (C_1)=1$.  Since $L_1 \cong \R$, then $L_1$ cannot be the identity group, so $\tau (L_1) =1$. All of the elements in $C_1$ are contained in $C$, and their limits under conjugacy by $P$ are contained in $L$.  Therefore $L_1$ is a tier 1 subgroup of $L$. 
  
   In general, $C$ has a flag of subgroups with tier $1,..., n-1$, as more of the entries on the diagonal are allowed to vary.  The conjugacy limits of this flag of subgroups of $C$ give a flag of conjugacy limits. 
  %An application of proposition \ref{rank} yields the result.   
  \end{proof}

  % The relation of being a conjugacy limit group introduces a partial order on the set of conjugacy limits of the Cartan subgroup. The length of any chain of conjugacy limit groups  of the  Cartan subgroup is at most $n$. 

Set $E \leq SL_8(\R)$ to be the image of the representation $\rho: \R^7 \to SL_8 (\R)$: 
  $$\rho( a,b,c,d,e,f,g)=
  \left( \begin{array}{cccccccc}
  1&0&0&0 & 0 & c& g &f\\
  0&1&0&0&c &b &f&e\\
  0&0&1&0 & b&a&e&d\\
  0&0&0&1&a&g&d&0\\
  0&0&0&0&1&0&0&0\\
  0&0&0&0&0&1&0&0\\
  0&0&0&0&0&0&1&0\\
  0&0&0&0&0&0&0&1
  \end{array} \right).
  $$

  It is easy to check that $E$ is an abelian subgroup, since matrix multiplication is given by 
  $$\left( \begin{array}{c|c}
  I & A \\
  \hline 
  0 &I
\end{array}   \right)
\left( \begin{array}{c|c}
  I & B \\
  \hline 
  0 &I
\end{array}   \right)
=
\left( \begin{array}{c|c}
  I & A+B \\
  \hline 
  0 &I
\end{array}   \right). 
  $$
  
  \begin{prop}\label{Eparam} The group $E$ has no 1 parameter subgroups of tier 1. 
  \end{prop}

\begin{proof} A matrix has rank 1 if and only if every $2 \times 2$ minor is zero.   We show that $\rho(a,b,c,d,e,f,g)-I_8$ has rank 1 if and only if $(a,b,c,d,e,f,g) =(0,...,0)$.  Consider the $2 \times 2$ minors of 
 $$ \left( \begin{array}{cccc}
 0 & c& g &f\\
c &b &f&e\\
 b&a&e&d\\
 a&g&d&0\\
 \end{array} \right) .$$
 
 Since the upper left minor must be zero, then $c =0$.  Looking at the minor directly below, implies $b=0$.  Continuing in this fashion, $b=0, a=0, d=0, e=0, f=0$ and $g=0$. (Alternatively, take all of the minors, and check $(0,0,...,0)$ is the only solution.)  Thus $\rho(a,b,c,d,e,f)-I_8$ has rank 1 if and only if $(a,b,c,d,e,f,g) =(0,...,0)$.  But if $(a,b,c,d,e,f,g) =(0,...,0)$ then $\rho(0, ..,0) - I_8$ is the zero matrix, with rank 0. Therefore $E$ (the image of $\rho$) contains no tier 1 subgroups.
\end{proof}

Combining Corollary \ref{ranklim} and Proposition \ref{Eparam}, shows the abelian group, $E$, is \emph{not} a conjugacy limit of $C$.  Thus there are two necessary conditions for a group to be a limit group: the group must be a flat group, and contain a tier 1 subgroup.  Are these conditions sufficient? 

Further, there are many more questions we might ask about the spaces $Red(n)$ and $Ab(n)$.  For example: are they connected?  Does every component of $Ab(n)$ contain a component of $Red(n)$, and is it possible to retract from $Ab(n)$ to $Red(n)$?  What properties characterize $Red(n)$ that are not inherited by $Ab(n)$?

\end{section}

\thanks{The author would like to thank \emph{Daryl Cooper} for many helpful discussions, and \emph{Thomas Haettel}, whose work inspired the paper.
The author was partially supported by NSF grants DMSÐ0706887, 1207068 and 1045292. 
The author acknowledges support from U.S. National Science Foundation grants DMS 1107452, 1107263,
1107367 RNMS: \emph{GEometric structures And Representation varieties (the GEAR Network).}}

%%%%%%%%%%%%%%%%%%%%%%%%%%%%%%%%%%%%%%%%%%%

  \end{document}